\documentclass[a4paper,10pt]{article}

\usepackage{tikz}

\usepackage{enumerate}
\usepackage{amsmath}

\usepackage{amsthm}

\usepackage{psfrag}

\usepackage{amssymb}
\usepackage{latexsym}

\newtheorem{Theorem}{Theorem}

\newtheorem{Definition}[Theorem]{Definition}
\newtheorem{Proposition}[Theorem]{Proposition}
\newtheorem{Lemma}[Theorem]{Lemma}
\newtheorem{Corollary}[Theorem]{Corollary}
\newtheorem{Remark}[Theorem]{Remark}

\def\P{{\mathcal P}}

\def\M{{\mathcal M}}

\def\ZZ{{\mathbb Z}}

\def\cc{{\mathbf c}}
\def\xx{{\mathbf x}_0}

\def\RR{{\mathbb R}}
\def\HH{{\mathbb H}}
\def\XX{{\mathbb X}}
\def\SS{{\mathbb S}}
\def\LL{{\mathbb L}}

\begin{document}

\title{Hyperbolic polygons of minimal perimeter in punctured discs}

\author{Joan Porti\footnote{ Partially supported  by grant MTM2015--66165--P (Mineco/FEDER)}
} 
\date{\today}

\maketitle

\begin{abstract}
 We prove that, among the polygons in a punctured disc with fixed angles, the 
perimeter is minimized by the polygon with an inscribed horocycle centered at 
the puncture. We 
generalize this to a disc with a cone point and to an annulus with a geodesic 
boundary component and a complete end. Then we apply this result to describe 
the 
minimum of 
the spine systole on the moduli space of punctured surfaces.
\end{abstract}

\section{Introduction}

% \blfootnote{Partially supported  
% by grant MTM2015--66165--P 
% (Mineco/FEDER)}
% %   \blfootnote
% %  {
% %  2010 MSC: 51M16, 57M20 }

Consider a complete hyperbolic disc with a puncture, i.e.\ with a cusp, 
$\XX_0=\HH^2/\langle 
\gamma_0\rangle$ where $\langle \gamma_0\rangle$ is the infinite cyclic group 
generated by a parabolic transformation $\gamma_0\in\operatorname{Isom}^+\HH^2$.
Fix $n\geq 1$ and $0<\beta_1,\ldots,\beta_n<\pi$  a family of angles.
Define 
$\P$ to be the space of polygons in $\XX_0$ with those 
(counterclockwise ordered) angles,
that separate both ends of $\XX_0$, and so that the cusp lies in the convex 
side 
of each angle. In Lemma~\ref{Lemma:existence} below we show that  $\P\neq 
0$, even for $n=1$. We prove:

\begin{Theorem}
\label{theoremc}
The unique minimum of the perimeter in $\mathcal P$ is realized by the polygon 
with an inscribed horocycle centered at the cusp.
\end{Theorem}

The case of a disc without any puncture (i.e.~the hyperbolic plane $\HH^2$)  
was considered in  
\cite{Porti}.  The
generalization in this paper is motivated by an application to spines of 
minimal 
length of hyperbolic surfaces.
A spine of a surface with finite topological type is a graph so that 
the 
surface retracts to it (for a closed surface one removes a point). Martelli, 
Novaga, Pluda, and Riolo \cite{Martelli} 
have shown  that for each \emph{closed} hyperbolic surface there 
are finitely many spines of minimal length, and their proof applies to the non 
compact case. Those spines are graphs with geodesic edges and  with 
trivalent 
vertices,
forming angles $2\pi/3$.

Let $\M_{g,p}$ denote the moduli space of a surface of genus 
$g$ with $p\geq 1$ punctures, with $p\geq 3$ when $g=0$.
The minimal length of a spine  is
called
the \emph{spine systole} of a surface and defines a function
$$
S\colon \M_{g,p}\to (0,+\infty).
$$
We see in Corollary~\ref{cor:proper} that $S$ is a proper function.

\begin{Corollary}
 \label{coro:spine} 
The minimum of $S\colon \M_{g,p}\to \RR$ is realized precisely by
subgroups of the modular group, i.e.~by surfaces 
$\HH^2/\Gamma$ with $\Gamma$ a subgroup of the congruence group $\Gamma(2)$.
 \end{Corollary} 
 
Here $\Gamma(2)$ denotes the congruence subgroup mod $2$ of  
$\mathrm{PSL}(2,\ZZ)$.
When $p=1$ those surfaces
are classically called \emph{cycloidal} 
\cite{Girondo,Millington,Petersson}. Surfaces 
$\HH^2/\Gamma$ with $\Gamma < \Gamma(2)$
 satisfy an extremal property:  
 there is a family of punctured horodiscs 
(i.e.~punctured discs 
in $\XX_0$ bounded by a horocycle), one for each 
cusp, whose interiors are embedded and pairwise disjoint, and whose 
complements are regions bounded by three horocyclic segments with tangent 
endpoints. 
In the cycloidal case ($p=1$) there is precisely a unique such a disc, which is 
maximal.
See 
\cite{Girondo} for extremality properties of embedded discs, punctured or not, 
as well as  \cite{Bavard}.

In Corollary~\ref{cor:length} we prove that $\min S=3(2 g-2+p)\log(3)$.

 \medskip
 
We shall consider a slightly more general situation, by 
replacing the cusp by a 
\emph{cone point} of angle $\alpha\in (0,2\pi)$ or a \emph{geodesic} of length 
$r>0$. Denote 
by $\XX$ this space, and denote by $\cc$ the cone point, the cusp, or the 
boundary component, 
according to the case we are considering.
Consider again $\P$ the space of polygons in $\XX$ with fixed angles
$0<\beta_1,\ldots,\beta_n<\pi$ that separate $\cc$ from the (infinite 
volume) end  of $\XX$ and so that 
$\cc$  lies in the convex side of each angle.
If $\cc$ is  a cone point of angle $\alpha$, then  we need to assume 
furthermore that 
\begin{equation}
 \label{eqn:hypothesis}
\alpha+\sum_{i=1}^n\beta_i < n\pi\, , 
\end{equation}
so that $\P\neq \emptyset$ (see Lemma~\ref{Lemma:existence}).

\begin{Definition}
\label{definition:equidistant}
An \emph{equidistant} to $\cc$ is the following curve in $\XX$: 
\begin{itemize}
 \item a horocycle centered at $\cc$ when it is a cusp,
 \item a circle centered at $\cc$ when it is a cone point, or
 \item a equidistant line to $\cc$ when it is a geodesic.
\end{itemize}
\end{Definition}

An equidistant has constant geodesic curvature $\kappa$,  where $\kappa =1$,
$\kappa>1$ or $\kappa<1$ 
in the respective cases of the definition. The following generalizes 
Theorem~\ref{theoremc}.

\begin{Theorem}
\label{theoremMain}
The unique minimum of the perimeter in $\mathcal P$ is realized by the polygon 
with an inscribed
equidistant to $\cc$.
\end{Theorem}

In $\HH^2$ a polygon is determined by the angles and edge 
lengths. In Lemma~\ref{Lemma:uniq} we prove that this is true also for 
polygons in $\P$, in 
particular the position with respect to $\cc$ is also determined by the angles 
and edge lengths.

The proof of Theorem~\ref{theoremc} uses techniques from 
\cite{Porti}, that relies on ideas introduced in \cite{Schlenker}, with 
some modifications.
The proof requires the Lorentz model of hyperbolic space so that several 
aspects of 
the three cases are unified. For instance  
$\cc$ is represented by a point $\xx$ in Lorentz space, which is lightlike for 
a 
cusp,
timelike for a cone point, and spacelike for a geodesic.

 Section~\ref{section:lorentz} is devoted to the  tools of Lorentz spaces we 
need. In Section~\ref{sec:polys} we construct the space of polygons $\P$ and we 
prove that it is a $(n-1)$-dimensional manifold. The main theorem is proved in 
Section~\ref{sec:main} and the corollary on spines is proved in 
Section~\ref{sec:spines}.

\paragraph{Acknowledgement} I am indebted to Christophe Bavard for pointing me 
to
the reference \cite{Gendulphe}

\section{Lorentz Space}
\label{section:lorentz}

The Lorentz space $\RR^2_1$ is $\RR^3$ equipped with the symmetric bilinear 
product with 
matrix
$$
J=\begin{pmatrix}
   -1 & 0 & 0 \\
   0 & 1 & 0 \\
   0 & 0 & 1
  \end{pmatrix}
$$
so that for $x,y\in\RR^2_1$, 
$
x\cdot y= x^t J y= -x^0 y^0 + x^1 y^1+x^2 y^2
$.
The Lorentz model of the hyperbolic plane is then
$$
\HH^2=\{ x\in\RR^2_1 \mid x\cdot x= -1,\ x^0> 0\} \, .
$$
From the equation $x\cdot x= -1$, the tangent space at a point is its 
orthogonal
$$
T_x\HH^2 = x^\perp = \{ y\in\RR^2_1 \mid x\cdot y= 0  \}  \,   .
$$
The de Sitter sphere is 
$$
\SS^2_1 =\{ x\in\RR^2_1 \mid x\cdot x= 1 \} \, .
$$ 
Every point $x\in \SS^2_1$ can be identified with an \emph{oriented line} in 
$\HH^2$
$$
\{y\in \HH^2\mid x\cdot y=0  \}\, .
$$
The orientation is provided by a normal vector.
Indeed, given $x\in \SS^2_1$, for  any point $p$ in 
the line $x$, $x$ can be viewed as a vector
in $T_p\HH^2$ (since $x\cdot p=0$) and $x$ is orthogonal to the line it 
represents.
We can also associate to $x$ a halfplane bounded by this line
$$
\{y\in \HH^2\mid x\cdot y\leq 0  \}\, .
$$

\begin{Remark}
 The vector $x\in\SS^2_1$ is the outwards normal field at the 
 boundary of the halfplane 
$
\{y\in \HH^2\mid x\cdot y\leq 0  \}\, .
$
\end{Remark}

To prove this remark, given a point $y\in \HH^2$ such that $x\cdot y=0$,  we 
consider the path $t\mapsto \varsigma(t) = y+t\, x+ O(t^2)$,
then $\varsigma(t)\cdot x= t+ O(t^2)$. Hence $\varsigma'(0)=x$ and the 
derivative of $\varsigma(t)\cdot x$ at $t=0$ is positive.

% \medskip

The light half-cone is
$$
\LL= \{ x\in\RR^2_1 \mid x\cdot x= 0, x^0>0  \}\, .
$$
Every $x\in \LL$ can be identified with the horocycle
$$
\{y\in\HH^2 \mid y\cdot x= - 1\}\, . 
$$
This is the boundary of the horodisc
$$
\{y\in\HH^2 \mid y\cdot x\geq  - 1\}\, . 
$$
On the other hand, the projective space on $\LL$ can be identified to the 
ideal boundary $\partial_\infty\HH^2$.

With the previous conventions, 
the 
Lorentz
product is related to the incidence, see 
\cite[Section~3.2]{Ratcliffe}:

\begin{Proposition}[Incidence and Lorentz product]
 \label{Proposition:incidence}
 \begin{enumerate}[(a)]
 \item Given two points $x,y\in \HH^2$ at distance $d\geq 0$, then
 $x\cdot y= -\cosh d$.
 
 \item Given a point $x\in \HH^2$ and an \emph{oriented line}  $y\in \SS^2_1$ 
 at   distance $d \geq 0 $, then
 $x\cdot y=  \pm \sinh d$, where the sign is negative if and only if $y$ 
belongs to the 
halfplane
 associated to $x$.

 \item The horocycle $x \in \LL$ is centered at an ideal endpoint of a 
line $y\in\SS^2_1$ if and only if $x\cdot y=0$.

 \item The horocycle $x \in \LL$ is tangent to the line $y\in \SS^2_1$ if and 
only if $x\cdot y=\pm 1$, with negative sign if the halfplane corresponding to 
$y$ contains the horodisc corresponding to $y$.
 
 \item If the oriented lines $x, y\in \SS^2_1$ are disjoint at distance 
 $d\geq 0$ ($d=0$ means that they are asymptotic), then 
 $
 x\cdot y = \pm \cosh d
 $, 
 where the sign 
 is positive when  the orientations are compatible (one of the 
halfplanes is contained  in the other).

 \item If the oriented lines $x, y\in \SS^2_1$ meet at one point with
 angle $\alpha$ (taking care of the orientations), then
 $
 x\cdot y = \cos \alpha
 $.
\end{enumerate}
\end{Proposition}

Following again \cite[Section~3.2]{Ratcliffe}  the \emph{Lorentzian cross 
product}  
$\boxtimes$ in $\RR^2_1$  is defined by the rule
$$
(u\boxtimes v)\cdot w=\det (u,v,w),\qquad \forall u,v,w\in\RR^2_1,
$$
where $\det (u,v,w)$ denotes the determinant of the matrix with entries the 
components of $u, v,w$. Namely 
$ u\boxtimes v = J(u\otimes v) $.
In particular $(\RR^2_1,\boxtimes)$ is a Lie algebra.

\begin{Remark} There is a natural  bijection $\RR^2_1 \leftrightarrow 
\mathfrak{so}(2,1)$ that is:
\begin{itemize}
 \item an isomorphism of Lie algebras  $(\RR^2_1,\boxtimes)\cong 
(\mathfrak{so}(2,1), [,]) $,
 \item an isomorphism of $\mathrm{SO}_0(2,1)$-modules, where the action on  
 $\RR^2_1$ is linear and on $\mathfrak{so}(2,1)$ is the adjoint,  and
 \item a Lorentz isometry, where  $\mathfrak{so}(2,1)$ is equipped with a 
multiple of the 
Killing form.
\end{itemize}
\end{Remark}

Now fix $\xx \in \HH^2$, $\SS^2_1$, or  $\LL$. Namely $\xx$ represents either 
a point in hyperbolic plane,  an oriented line, or an ideal point (viewed 
projectively).
Let $e_1,\ldots,e_n\in \SS^2_1$ be a collection of oriented lines.

\begin{Lemma}
\label{Lemma:equidistant}
 The oriented lines $e_1,\ldots,e_n\in \SS^2_1$ are tangent to a equidistant to 
$\xx$ 
 if and only if
 $$
 \vert e_1\cdot\xx \vert =\cdots = \vert e_n\cdot\xx\vert = \textrm{ctnt}\, .
%  \begin{cases}
%   < 0 & \textrm{if } \xx\cdot\xx =-1 \\
%   >0 & \textrm{if } \xx\cdot\xx = 0 \\
%   > 1 & \textrm{if } \xx\cdot\xx  = 1
%  \end{cases}
 $$
 In addition, the absolute values can be removed by taking care of  
orientations.
\end{Lemma}

\section{The space of polygons}
\label{sec:polys}

Let $\P$ denote the space of polygons in $\XX$ as in the introduction.
It can be embedded in $T^1\XX\times \RR^n$  by looking at the tangent vector to 
a given edge at one of its vertices,  and the edge lengths $l_1,\ldots,l_n>0$.

By convexity, the closure $\overline\P$ is obtained by considering edges of 
length zero or, when $\cc$ is a cone point, by allowing a vertex or the 
interior 
of an edge to meet the cone point. In this case, $\alpha>\pi$ when $\cc$ meets  
the interior 
of an edge, or
  $\alpha+\beta_i> 2\pi$ 
when $\cc$ meets the $i$-th vertex.

As before, $\xx\in \HH^2$ when $\cc$ is a cone point, $\xx\in \SS^2_1$ when 
$\cc$ is geodesic, and  $\xx\in \LL$ when $\cc$ is a cusp.

Fix $e_0$ an oriented  line so that $e_0\cdot\xx=0$
and fix a point  $p_0\in e_0$ in this line. Let $g\colon (-\infty,\infty)\to 
\HH^2$ denote a  parametrization of $e_0$ so that $g(0)=p_0$ and
  $\{e_0, \dot{g}(0)\}$ is a positive frame in $T_{p_0} \HH^2$. In addition, 
assume:
  \begin{itemize}
   \item When $\xx\in\HH^2$, then $\xx=p_0$.
   \item When $\xx\in\SS^2_1$, then $\xx\cap e_0= \{p_0 \}$ and $\{e_0,\xx\}$ 
is a 
    positive frame (i.e.~$\dot{g}(0)=\xx $).
   \item When $\xx\in \LL$, then $\xx\cdot p_0=1$ and $g(-\infty)$ is the
   projective class of $\xx$.
  \end{itemize}
% 
%  
%  
%  when $\xx\in\SS^2_1$, 
% $p_0\in\xx$ and $\{\xx,e_0\}$ is a positive frame, and when 
% $\xx\in \LL$, $\xx\cdot p_0=1$ and $g(-\infty)$ is the projective class of 
% $\xx$.
% 
% 
Consider also an orientation preserving isometry $\gamma\in \mathrm{SO}_0(2,1)$ 
as follows:
 \begin{itemize}
  \item When $\xx\in\HH^2$, $\gamma$ is a (positively oriented) rotation of angle $\alpha\in (0,2\pi)$ around 
$\xx$.
  \item When $\xx\in\SS^2_1$, $\gamma$ is a loxodromic isometry with axis
$\xx$ of translation length $r$  (in the direction $-e_0$).
  \item When $\xx\in\LL$, $\gamma$ is a parabolic transformation than 
         fixes $\xx$ (in the direction $-e_0$).
 \end{itemize}
 Chose 
$l_1,\ldots,l_n\in [0,+\infty)$ the lengths of the 
sides of the polygon. We shall also consider two parameters $l_0\in\RR$ and 
$\theta \in [0,2\pi]/\{0\sim2\pi\} \cong S^1$. 

We define maps $v$ and $w$ from the parameter spaces
 to the 
unit tangent bundle $T^1\HH^2$ as follows. Start with the vector 
$\dot{g}(l_0)\in T _{g(l_0)}\HH^2$ and rotate it by an angle $\theta$, call 
this vector $v(l_0,\theta)$. This defines a map:
$$
v \colon \RR\times S^1\to T^1\HH^2 \, .
$$
Then consider a polygonal path starting at  $q_0=g(l_0)$ in the direction of 
$v(l_0,\theta)$ that is the union of $n$ segments of lengths 
$l_1,l_2,\ldots,l_n$ with ordered angles $\beta_1,\beta_2,\ldots,\beta_{n-1}$,
so that at the end of the $i$-th edge turns left by the exterior angle 
$\pi-\beta_i$ and continue to the $(i+1)$-th edge.
At the end of the $n$-th edge, consider the tangent unitary 
vector defining an angle $\beta_n$, i.e.\ turn left by the exterior angle 
$\pi-\beta_n$.
This defines a map
$$
w\colon \RR\times S^1\times \RR^n\to T^1\HH^2.
$$
Then 
$\overline \P$ is contained in the set 
$$
\overline \P\subseteq
\{
(l_0,\theta,l_1,\ldots,l_n)\in I\times S^1\times (0,+\infty)^n\mid 
w(l_0,\theta,l_1,\ldots,l_n)= \gamma v(l_0,\theta)
\}
$$
where 
\begin{itemize}
 \item $I=[0,+\infty)$ when $\cc$ is a cone point,
 \item $I=(0,+\infty)$ when $\cc$ is a geodesic, and
 \item $I=\RR$ when $\cc$ is a cusp.
\end{itemize}
When $\cc$ is a cone point of angle $\alpha<\pi$ or $\cc$ is a geodesic, then 
$l_0=0$ is not possible by convexity.

\begin{Proposition}
\label{Proposition:tangent}
 The space $\P$ is a $(n-1)$-dimensional analytic manifold with tangent space 
 at a point $p\in\P$:
 $$
 T_p\P=\left\{ (\dot{l}_0, \dot{\theta}, \dot{l}_1,\ldots,  \dot{l}_n  ) 
 \in\RR^{n+2} \mid
  \dot{l}_0 (1-\gamma) e_0+\dot{\theta} (1-\gamma)  q_0+\sum_{i=1}^n  \dot{l}_i 
  e_i = 0 
 \right\}\, .
 $$
\end{Proposition}

The unit tangent bundle $T^1\HH^2$ is naturally identified to the isometry 
group  $\mathrm{SO}_0(2,1)$, as the action is simply transitive. Thus the 
tangent space at a given point is naturally identified with $ 
\mathfrak{so}(2,1)\cong \RR^2_1$.
In the next lemma 
the Lie algebras correspond to the tangent space at different points.

\begin{Lemma} The tangent map  $w_*\colon \RR^{n+2}\to 
T_{w(l_0,\theta,l_1,\ldots,l_n)} 
\HH^2 \cong \mathfrak{so}(2,1)$ satisfies 
 $ w_* \left(\frac{\partial\phantom{l_i} }{\partial l_i}\right)= e_i$ for 
$i=0,\ldots,n$
 and $w_* 
\left(\frac{\partial\phantom{\theta} }{\partial\theta}\right)= 
g(l_0)=q_0$.

 The tangent map $v_*\colon \RR^{2}\to  
 T_{w(l_0,\theta)} 
\HH^2 \cong
 \mathfrak{so}(2,1)$
satisfies
 $ v_* \left(\frac{\partial\phantom{l_0} }{\partial l_0}\right)= e_0$
 and $v_* 
\left(\frac{\partial\phantom{\theta} }{\partial\theta}\right)= 
g(l_0)=q_0$.
\end{Lemma}

\begin{proof}
 Increase one of the $l_j$ by keeping the other $l_k$ and $\theta$ 
 constant means composing the map (either $w$ or $v$) with an isometry with 
axis $e_i\in\RR^2_1$,
 and its derivative corresponds to   
 $e_i\in  \mathfrak{so}(2,1)  $ after the previous identifications of the 
tangent space to $T^1\HH^2$ to
 $\mathfrak{so}(2,1) \cong \RR^2_1$. The same argument applies to $\theta$. 
 \end{proof}

\begin{Lemma}
\label{Lemma:image1-gamma}
 Assuming that $l_0>0$ when $\cc$ is a cone point, we have
 $$
 \langle (1-\gamma) e_0,(1-\gamma) q_0\rangle = \xx^\perp \, ,
 $$
 where $q_0=g(l_0)$.
\end{Lemma}

\begin{proof}
 We start checking that, for the different possibilities of $\xx$,   
 \begin{equation}
\label{eqn:directsum}
 \langle e_0, q_0, 
\xx\rangle = \RR^2_1\, .
  \end{equation}
 Namely, when $\xx$ is a horocycle, we may assume up to isometry that
 $$
 \xx=\begin{pmatrix}
      1 \\ 1 \\ 0
     \end{pmatrix}, \quad
 e_0= \begin{pmatrix}
      0 \\ 0 \\ 1
     \end{pmatrix}, \quad
 q_0=  \begin{pmatrix}
      \cosh(t) \\ \sinh(t) \\ 0
     \end{pmatrix}, \quad   
 $$
 for some $t\in \RR$. When $\xx$ is a geodesic, 
  $$
 \xx=\begin{pmatrix}
      0 \\ 1 \\ 0
     \end{pmatrix}, \quad
 e_0= \begin{pmatrix}
      0 \\ 0 \\ 1
     \end{pmatrix}, \quad
 q_0=  \begin{pmatrix}
      \cosh(t) \\ \sinh(t) \\ 0
     \end{pmatrix}, \quad   
 $$
 for some $t>0$. When $\xx$ is a point in hyperbolic plane,
 since we assume $l_0>0$, 
   $$
 \xx=\begin{pmatrix}
      1 \\ 0 \\ 0
     \end{pmatrix}, \quad
 e_0= \begin{pmatrix}
      0 \\ 0 \\ 1
     \end{pmatrix}, \quad
 q_0=  \begin{pmatrix}
      \cosh(t) \\ \sinh(t) \\ 0
     \end{pmatrix}, \quad   
 $$
 for some $t>0$. This establishes~\eqref{eqn:directsum}. 
Then, since $\ker(1-\gamma)=\langle\xx\rangle $ and $ \gamma$ 
is an isometry,
the lemma follows.% from \eqref{eqn:directsum}.
\end{proof}

\begin{proof}[Proof of Proposition~\ref{Proposition:tangent}]
 Consider $M$ the  matrix of size $3\times (n+2)$ with columns 
 $M_1,\ldots, M_{n+2}$, where 
\begin{equation}
 \label{eqn:MatrixM}
  M_1 =(1-\gamma) e_0, \quad  M_2= (1-\gamma)  q_0, \quad M_3= e_1, 
  \ldots, \quad   
  M_{n+2}=e_n \, .
\end{equation}
 We aim to show that $\operatorname{rank}(M)=3$, so that the maps $w$ and 
$\gamma v$ 
are transversal.
 Assume first that in the elliptic case $l_0>0$. By 
Lemma~\ref{Lemma:image1-gamma},  
 it suffices to have that $e_i\cdot \xx\neq 0$ for some $i= 1,\ldots,n $. By 
the incidence relations, Proposition~\ref{Proposition:incidence}, it is 
impossible that  $e_i\cdot \xx = 0$ for all 
$i=1,\ldots,n$ (e.g.\ when there is a cusp this would mean that all edges 
belong to a geodesic ending at the cusp, and similarly for the other cases). 
 
When $l_0=0$ in the elliptic case, $q_0=\xx$ hence $(1-\gamma) q_0=0$. In this
case, since  $\xx=\gamma \xx$ is the starting  and final point on the 
polygonal path, it is a closed polygon in $\HH^2$. In particular the number of 
edges is $\geq 3$ and they are generic enough so that $e_1,\ldots,e_n$ are 
linearly independent in $\RR^2_1$.
\end{proof}

\begin{Remark}
 \label{Remark:outofP}
 The proof of Proposition~\ref{Proposition:tangent} yields that  
$\overline \P$ is contained in a smooth manifold of the same dimension as 
$\P$, with the tangent space described by 
Proposition~\ref{Proposition:tangent}. We shall use this to integrate tangent  
vectors into deformations of polygons.
\end{Remark}

\section{Proof of the main theorem}
\label{sec:main}

The proof of Theorem~\ref{theoremMain} follows from the following 4 lemmas.

\begin{Lemma}
 \label{Lemma:properness}
 The perimeter $\overline{\P}\to [0,+\infty)$ is a proper function. 
\end{Lemma}

\begin{Lemma}
 \label{Lemma:critical}
 A polygon in $\P$ is a critical point of the perimeter iff it has an inscribed 
equidistant.
\end{Lemma}

\begin{Lemma}
 \label{Lemma:totheinterior}
A polygon in  $\overline{\P}-\P$ can be perturbed to $\P$ while decreasing the 
perimeter. 
\end{Lemma}

\begin{Lemma}
 \label{Lemma:existence}
There exists a unique polygon in $\P$ with an inscribed equidistant.
\end{Lemma}

\begin{proof}[Proof of Lemma~\ref{Lemma:properness}]
Seeking a contradiction, assume that 
we have a sequence of parameters in $\overline{\P}$ with $l_0\to+\infty$
but  $l_1,\ldots,l_n\geq 0$ are bounded. This is not possible because the 
distance 
between $g(l_0)$ and $\gamma g(l_0)$  converges to infinity as $l_0\to+\infty$, 
but this distance is 
bounded by the perimeter
$l_1 +\cdots+ l_n$. This establishes properness when 
 $\cc$ is a cone point or a geodesic.
When $\cc$ is a cusp, 
there could be a sequence of polygons 
with $l_0\to-\infty$, while $l_1,\ldots,l_n\geq 0 $ are bounded. This implies 
that the 
sequence of polygons are contained in horodiscs with area going to zero, but 
this contradicts Gauss-Bonnet theorem: the area depends only on the angles
$\beta_1,\ldots,\beta_n$.
\end{proof}

\begin{proof}[Proof of Lemma~\ref{Lemma:critical}]
 Being a critical point means that whenever $\dot l_0$, $\dot \theta$, 
 $\dot l_1,\ldots, \dot l_n$ satisfy
 $$
   \dot{l}_0 (1-\gamma) e_0+\dot{\theta} (1-\gamma)  q_0+
   \dot{l}_1 e_1 + \cdots + \dot{l}_n e_n  
   = 0\, ,
 $$
 then $\dot{l}_1+\cdots +\dot{l}_n=0$. Let $M$ be the matrix of 
 size $3\times (n+2)$ defined by columns as in Equation~\eqref{eqn:MatrixM},
 in the proof of Proposition~\ref{Proposition:tangent}. By the proof of the 
 same proposition, $\operatorname{rank}(M)=3 $. Let $\overline M$ be the matrix
 of size $4\times (n+2)$ obtained by adding the row
 $$
  \begin{pmatrix}
   0 & 0 & 1 & 1 & \cdots & 1
  \end{pmatrix}
 $$
 to the bottom of $M$. Being a critical point means that $\ker M= \ker\overline 
M$, 
i.e.\ that $\operatorname{rank}(\overline M)=3$.
Set 
$$
\begin{pmatrix}
 z^0 \\ z^1 \\ z^2 \\ 1 
\end{pmatrix}
\in \ker  (\overline {M}^t)
\quad \text{ and } \quad
z= 
\begin{pmatrix}
- z^0 \\ z^1 \\ z^2  
\end{pmatrix}\neq 0\, .
$$
By hypothesis
$$
((1-\gamma) e_0)\cdot z = 
((1-\gamma) q_0)\cdot z = 0.
$$
Thus, by Lemma~\ref{Lemma:image1-gamma} $z$ is a multiple of $\xx$: $z = 
\lambda \xx$ for some $\lambda\in \RR\setminus\{0\}$.
Hence 
$$
l_1\cdot \xx = \cdots = l_n\cdot \xx= -1/ \lambda.
$$
By Lemma~\ref{Lemma:equidistant}, and discarding the values of  $\lambda$ 
that contradict convexity, the lemma is proved.
\end{proof}

\begin{proof}[Proof of Lemma~\ref{Lemma:totheinterior}]
We consider first the case where the cone point $\cc$ meets 
a single vertex, say the first one.
By convexity,  
$\beta_n+\alpha > 2\pi$. By the previous construction 
$l_0=0$, and we aim to deform the parameters so that $l_0$ increases but the 
perimeter 
decreases.
When $l_0=0$, deforming $\theta$ does not change the resulting polygon. Thus 
we chose $\theta$  so that the line that bisects $e_0$ and $\gamma e_0$
is the same that bisects $e_1$ and $e_n$ but the corresponding half-lines are 
opposite, see Figure~\ref{fig:lines0}.

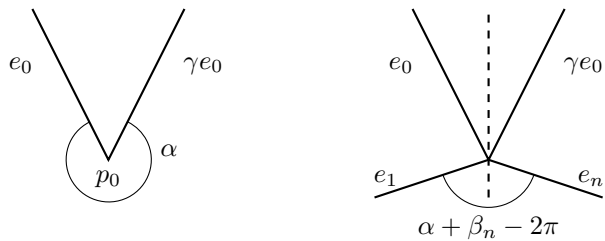
\begin{figure}
\begin{center}
\begin{tikzpicture}[scale=.5]
 \draw [thick] (-2,4)-- (0,0) --(2,4);
 \draw [thin] (-.5,1) arc [radius=1.12, start angle=117, end angle= 423];
 \node at (1.6,.25) {$\alpha$};
 \node at (-2.3,2.5) {$e_0$};
 \node at (2.5,2.5) {$\gamma e_0$};
 \node at (0,-.5) {$ p_0$};
 \draw [thick] (8,4)-- (10,0) --(12,4);
  \draw [thick] (7,-1)-- (10,0) --(13,-1);
  \draw [thick, dashed] (10,-1)--(10,4);
\draw  [thin] (8.8,-.4) arc  [radius=1.26, start angle=198.4, end angle= 341.6];
 \node at (7.7,2.5) {$e_0$};
\node at (12.5,2.5) {$\gamma e_0$};
 \node at (7.3,-.5) {$e_1$};
\node at (12.7,-.5) {$e_n$};
\node at (10,-1.8) {$ \alpha+\beta_n-2\pi $};
 \end{tikzpicture}
% {\psfrag{e}{$e_0$}
% \psfrag{e1}{$e_1$}
% \psfrag{en}{$e_n$}
% \psfrag{g}{$\gamma_0e_0$}
% \psfrag{q}{$p_0=q_0$}
% \psfrag{a}{$ \alpha$}
% \psfrag{b}{$ \alpha+\beta_n-2\pi$}
% \includegraphics[height=3cm]{lines0.eps}
% }
\end{center}
   \caption{The relative position of the lines $e_0$ and $\gamma e_0$ and the 
vertex of the polygon
}\label{fig:lines0}
\end{figure}

\begin{figure}
\begin{center}
\begin{tikzpicture}[scale=.5]
 \draw [thin] (-2,4)-- (0,0) --(2,4);
   \draw [thin] (-3,-1)-- (0,0) --(3,-1);
   \draw [very thick] [<->] (-1,2)--(0,0)--(1,2);
%       \draw [very thick] [<->] (-1.5,-.5)--(0,0)--(1.5,-.5);
      \draw [very thick] [<-] (-1.5,-.5)--(0,0);
      \draw [very thick] [<-] (0,0)--(1.5,-.5);
 \node at (-2.3,2.5) {$e_0$};
 \node at (2.5,2.5) {$\gamma e_0$};
  \node at (-2.7,-.5) {$e_1$};
\node at (2.7,-.5) {$e_n$};

 \draw [very thick] [->] ((10,0)--(9,2);
 \draw [very thick] [->] ((10,0)--(9,-2);
 \draw [very thick] [->] ((10,0)--(8.5,-.5);
 \draw [very thick] [->] ((10,0)--(8.5,.5);
 \node at (7.7,2) {$e_0$};
\node at (7.7,-2) {$-\gamma e_0$};
 \node at (7.3,-.5) {$e_1$};
\node at (7.3,.5) {$e_n$};
 \end{tikzpicture}
\end{center}
   \caption{The oriented lines viewed as  vectors in $T_{p_0}\HH^2$. 
}\label{fig:lines1}
\end{figure}
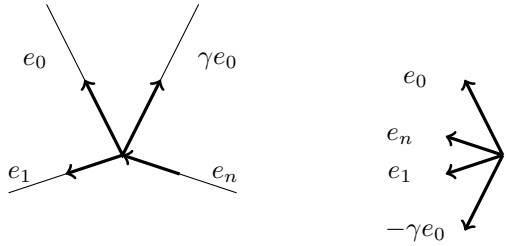

Since $\xx=p_0=q_0$ belongs to the lines $e_0$, $\gamma e_0$, $e_1$ and $e_n$, 
we view   them as tangent 
vectors 
to $\xx$, i.e.~they lie in the plane $T_{\xx}\HH^2$.
Now $e_0 - \gamma e_0$ and $e_1+e_n$ are both tangent vectors perpendicular to 
the bisector, and they both  point in the same direction (see 
Figure~\ref{fig:lines1}):
$$
(1-\gamma) e_0= \lambda (e_1+e_n)\qquad \textrm{ for some } \lambda> 0.
$$
Hence we may consider a deformation tangent to the vector
$\dot{l}_0=1$, $\dot{\theta}=0 $, $\dot l_1=\dot l_n=-\lambda$, $\dot l_2=
\dot l_3=\cdots =\dot l_{n-1}=0$. This is a vector tangent to the manifold 
in the equations defined in Proposition~\ref{Proposition:tangent}, and we have 
shown that this is a smooth point, see Remark~\ref{Remark:outofP}. Hence the 
tangent vector corresponds to a 
deformation, and by construction it  pushes the cone point to the 
interior of the polygon ($\dot l_0>0$) and the derivative of the perimeter is
\[
\dot l_1+\cdots +\dot l_n= - 2 \lambda <0\, . 
\]

 When $\cc$ meets the interior of an edge, the proof  is 
analogous by viewing an interior point as a vertex of angle $\pi$.
When some of the $l_i$ vanishes, this is precisely the content of Lemma~11 in 
 \cite{Porti}.
In general, the  tangent vectors to deformations can be added in 
order to combine the different deformations, namely pushing the cone point 
away from the polygon and  increasing the length of edges of length 0 in the 
same deformation, using again Remark~\ref{Remark:outofP}.
\end{proof}

\begin{proof}[Proof of Lemma~\ref{Lemma:existence}]
 The existence and uniqueness is proved by gluing certain polygons. 
 
 Assume first that $\cc$ is a cusp. For each vertex $i$, consider an ideal 
hyperbolic triangle with angles $0$, $\pi/2$, and $\beta_i/2$. Double this 
triangle by a reflection on the edge opposite to the right angle,
obtaining a quadrilateral with angles $0$, $\pi/2$, $\beta_i$, and $\pi/2$,
see Figure~\ref{fig:quad0}.
The angles do not determine this quadrilateral, there are quadrilaterals that 
are non symmetric, but 
this is the only one whose finite edges are tangent to an horocycle centered at 
the ideal point.
From those quadrilaterals one can construct the polygon, and it is unique by 
the 
tangency to the horocycle.

\begin{figure}[h]
\begin{center}
\begin{tikzpicture}[scale=.5]
 \draw [very thick] (0,0) arc [radius=20, start angle=270, end angle= 296];
  \draw [very thick] (10.35,0) arc [radius=6, start angle=231, end angle= 206];
\draw [very thick] (0,0) arc [radius=20, start angle=90, end angle= 64];
  \draw [very thick] (10.35,0) arc [radius=6, start angle=129, end angle= 154];
  \draw [thin] (8.45,1.8)-- (8.60,1.55) --(8.85,1.7);
    \draw [thin] (8.45,-1.8)-- (8.60,-1.55) --(8.85,-1.7);
    \node at (-.5,0) {$0$};
        \node at (11,0) {$\beta$};
\draw [thin] (10,-.3) arc [radius=.5, start angle=220, end angle=140];
  \draw [thin] (8.72,-2) arc [radius=4.75, start angle=-25, end angle= 25];
 \end{tikzpicture}
\end{center}
   \caption{The quadrilateral in the proof of 
Lemma~\ref{Lemma:existence} when $\cc$ is a cusp, with the 
inscribed horocycle.}\label{fig:quad0}
\end{figure}
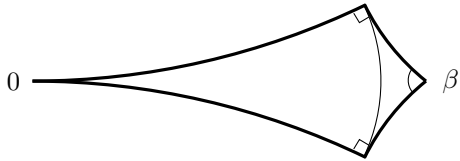

When $\cc$ is a cone point, given $\alpha_i$ the building block is a triangle 
with angles $\alpha_i/2$, $\pi/2$ and $\beta_i/2$. Double it along the 
long 
edge, to get a quadrilateral with angles $\alpha_i$, $\pi/2$, $\beta_i$, and  
$\pi/2$, see Figure~\ref{fig:quad1}, so that the edges that meet at angle 
$\beta_i$ 
are tangent to circle centered at the vertex of angle $\alpha_i$. Let 
$r(\alpha_i,\beta_i)$ denote the 
radius of this circle, which is the length of the two edges adjacent to 
the vertex with angle $\alpha_i$. For fixed $\beta_i$ the
radius $r(\alpha_i,\beta_i)  $ is strictly decreasing on $\alpha_i$, with 
 $r(\pi-\beta_i,\beta_i)=0$  and $r(0,\beta_i)=+\infty$. Thus by gluing the 
blocs one can realize any cone angle $< n\pi-\beta_1-\cdots -\beta_n$, in 
particular
$\alpha$ by Assumption~\eqref{eqn:hypothesis}. Uniqueness also follows.

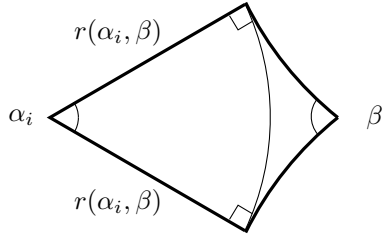
\begin{figure}[h]
\begin{center}
\begin{tikzpicture}[scale=.75]
 \draw [very thick]  (8.72,-2) -- (5.3,0) -- (8.72,2);
 \draw [very thick] (10.35,0) arc [radius=6, start angle=231, end angle= 206];
  \draw [very thick] (10.35,0) arc [radius=6, start angle=129, end angle= 154];
  \draw [thin] (8.45,1.8)-- (8.60,1.55) --(8.85,1.7);
    \draw [thin] (8.45,-1.8)-- (8.60,-1.55) --(8.85,-1.7);
        \node at (11,0) {$\beta$};
\draw [thin] (10,-.3) arc [radius=.5, start angle=220, end angle=140];
  \draw [thin] (8.72,-2) arc [radius=4.75, start angle=-25, end angle= 25];
    \draw [thin] (5.8,0) arc [radius=.5, start angle=0, end angle=33];
     \draw [thin] (5.8,0) arc [radius=.5, start angle=0, end angle=-33];
     \node at (4.8,0) {$\alpha_i$};
     \node at (6.5,1.5) {$r(\alpha_i,\beta)$};
         \node at (6.5,-1.5) {$r(\alpha_i,\beta)$};
 \end{tikzpicture}
% {\psfrag{r}{$r(\alpha_i,\beta_i) $}
% \psfrag{b}{$\beta_i$}
% \psfrag{t}{$\alpha_i$}
% \includegraphics[height=2.5cm]{quadrilateral1.eps}
% }
\end{center}
   \caption{The quadrilateral in the proof of 
Lemma~\ref{Lemma:existence} when  $\cc$ is a cone point, with 
the inscribed circle.}\label{fig:quad1}
\end{figure}

When $\cc$ is a geodesic, the building blocks are similar: symmetric 
pentagons with four right angles and   one angle $\beta_i$ (that is the 
double of a quadrilateral with three right angles and one angle $\beta_i/2$),
Figure~\ref{fig:quad2}.
The argument now is similar, as $r=r(d_i,\beta_i)$ is a strictly decreasing 
function on the length $d_i$ of the segment opposite to $\beta_i$, 
$r(+\infty, \beta_i)=0$ (approaching a triangle with two ideal vertices), and 
$r(0, \beta_i)=+\infty$ (approaching a triangle with one ideal vertex).
\end{proof}

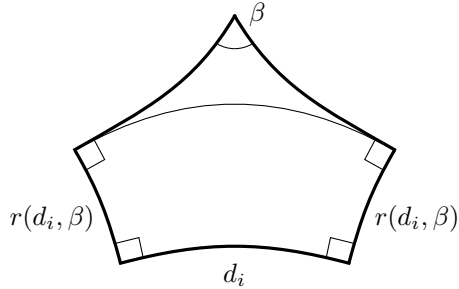
\begin{figure}[h]
\begin{center}
\begin{tikzpicture}[scale=.3]
 \draw [very thick] (-5.07,0) to [out=15,in=165] (5.07,0);
 \draw [very thick] (5,0) to [out=75,in=240] (7,5);
 \draw [very thick] (7.07,5) to [out=150,in=-60]  (-0.03,11);
 \draw [very thick] (-5,0) to [out=105,in=300] (-7,5);
   \draw [very thick] (-7.07,5) to [out=30,in=240]  (0.03,11);
   \draw [thin] (-7,5) to [out=30,in=150]  (7,5);
   \draw [thin] (4,0.2)--(4.2,1.15)--(5.2,.9) ;
   \draw [thin] (-4,0.2)--(-4.2,1.15)--(-5.2,.9) ;
   \draw [thin] (6.15,5.45)--(5.7,4.6)--(6.5,4.15) ;  
   \draw [thin] (-6.15,5.45)--(-5.7,4.6)--(-6.5,4.15) ; 
   \draw [thin] (0,9.5) arc [radius=1.5, start angle=270, end angle=300];
   \draw [thin] (0,9.5) arc [radius=1.5, start angle=270, end angle=240];
   \node at (0,-.5) {$d_i$};
   \node at (8,2) {$r(d_i,\beta)$};   
   \node at (-8,2) {$r(d_i,\beta)$};   
   \node at (1,11) {$\beta$};
 \end{tikzpicture}
\end{center}
   \caption{The pentagon in the proof of 
Lemma~\ref{Lemma:existence} when  $\cc$ is a geodesic, with 
the inscribed equidistant.}\label{fig:quad2}
\end{figure}

This concludes the proof of Theorem~\ref{theoremc}. Notice that 
Lemma~\ref{Lemma:existence} also establishes that $\P$ is non empty. The proof 
of  Theorem~\ref{theoremc} also shows that $\P$ has dimension $n-1$. One may 
still ask whether the edge lengths and angles determine a polygon in $\P$, as 
there 
are two further parameters that determine the position relative to $\cc$. 

\begin{Lemma} 
\label{Lemma:uniq}
A polygon in $\P$ is determined by its edge lengths $l_1,\ldots, l_{n}> 0$
and angles $\beta_1,\ldots,\beta_n\in (0,2\pi)$.
In particular its position relative to $\cc$ is determined by the 
lengths and 
the angles. 
\end{Lemma}

\begin{proof}
We unfold the polygon in $\HH^2$: namely we consider  a piecewise geodesic path 
consisting 
of $n$ segments of lengths $l_1,\ldots, l_{n}> 0$ and angles 
$\beta_1,\ldots,\beta_{n-1}\in (0,2\pi)$.
When $c$ is a cusp, the lemma follows because there exists a unique 
\emph{oriented} parabolic isometry  that joins the endpoints of this path. In 
fact, without taking into account the orientation  there are 
two of them, but if we want the cusp to be in the convex side there is only 
one choice (two different points in $\HH^2$ can be joined by precisely two 
curves of constant geodesic curvature $1$). This establishes the lemma when 
$\cc$ is a cusp. Notice that  joining 
the endpoints by a parabolic isometry is a necessary condition, but not 
sufficient. 
The  proof when $\cc$ is a cone point or a geodesic is analogous, instead of 
parabolic isometries one must consider rotations of given angle, or loxodromic 
elements of given translation length, respectively. 
\end{proof}

\section{Spines of minimal length}
\label{sec:spines}

Let $F$ be a non compact, complete, and  orientable  hyperbolic surface with 
finite topology. As 
said in the introduction, a spine is a graph in $F$ so that $F$ retracts to it, 
and 
the proof of Martelli, Novaga, Pluda, and Riolo  \cite{Martelli}  
in the compact case yields 
the existence of  spines of minimal length. Those are piecewise geodesic graphs 
with 
trivalent vertices, so that the angles are $2\pi/3$.

\begin{proof}[Proof of Corollary~\ref{coro:spine}]
The endpoints of surfaces in $\M_{g,p}$ are cusps, recall that $p\geq 1$. If 
we 
cut open a surface in  $\M_{g,p}$ along a spine of minimal length, then we 
obtain  polygons with angles $2\pi/3$ in punctured discs, one for each end of 
the surface.
Since the perimeter is minimized by the  polygon with an inscribed horocycle, 
this 
length is minimized precisely by surfaces obtained from these polygonal 
domains (that in particular are regular).  
Thus surfaces minimizing the spine systole are an orbifold covering of 
the 2-sphere with 
a puncture and two 
cone points of order $2$ and $3$ respectively, namely the modular orbifold 
$\HH^2/\mathrm{PSL}(2,\ZZ)$. Therefore the surfaces that minimize the spine 
systole are   $\HH^2/\Gamma$ for 
some $\Gamma < \mathrm{PSL}(2,\ZZ)$. In fact $\Gamma<\Gamma(2)$, see for 
instance the proof of \cite[Proposition~A.4]{Gendulphe}.

On the other hand, since the modular orbifold $\HH^2/ \mathrm{PSL}(2,\ZZ)$ 
has a horodisc centered at the cusp whose interior is properly 
embedded and its closure has self-intersection precisely at the cone point of 
order $2$, every modular surface is obtained from punctured polygonal domains 
with an inscribed horocycle as above.
\end{proof}

The edge length of the  polygon of angles $\beta$ in a punctured disc with an 
inscribed horocycle is 
\begin{equation}
2 \log\frac{1+\cos\beta/2}{\sin(\beta/2)}= 2\sinh^{-1} (\cot(\beta/2)) 
\end{equation}
independently of the number of edges.
For spines of minimal length, we are interested in $\beta=2\pi/3$. This yields
\begin{equation}
\label{eqn:2pi3}
2\log\frac{1+1/2}{\sqrt{3}/2}=\log (3 ).
\end{equation}
Thus we have:

\begin{Corollary}
\label{cor:length}
 Let $F$ be an orientable  hyperbolic surface  with finite topology, of genus 
$g$ and with $p\geq  1$ ends. Then the length $l$ of a spine in $F$ satisfies
$$
l\geq 3 (2 g+p-2)\log({3}),
$$
with equality if and only if $F=
 \HH^2/\Gamma$ for 
some $\Gamma < \Gamma(2) < \mathrm{PSL}(2,\ZZ)$
and the spine 
has minimal length. 
\end{Corollary}

\begin{proof}
For a general surface, as its retraction to its convex core is distance 
decreasing, we may assume that $F$ is a surface with boundary components and 
cusps. Using the constructions of Lemma~\ref{Lemma:existence}, e.g.\ 
Figures~\ref{fig:quad0} and \ref{fig:quad2}, the length of a regular polygon 
with
angles $2\pi/3$ is bounded below by the cusped case, and   the minimum is 
realized by surfaces  $\HH^2/\Gamma$ for 
some $\Gamma < \Gamma(2) < \mathrm{PSL}(2,\ZZ)$. As a minimal  spine is a 
trivalent graph, the number of edges is $ -3\chi(F)= 3 ( 2 g + p-2) $. Hence 
the 
corollary follows from~\eqref{eqn:2pi3}.
\end{proof}

Finally, a spine of minimal length may be nonunique, but from 
Lemma~\ref{Lemma:uniq}, we deduce:

\begin{Corollary}
\label{cor:proper}
A surface is uniquely determined by the spine of minimal length. In particular 
$S\colon \M_{g,p}\to (0,+\infty)$ is proper.
\end{Corollary}

\begin{footnotesize}
\bibliographystyle{plain}

% \bibliography{polygons}

\noindent \textsc{Departament de Matem\`atiques, Universitat Aut\`onoma de 
Barcelona,\\  08193 Cerdanyola del Vall\`es (Spain)}

\noindent 
and 

\noindent 
\textsc{Barcelona Graduate School of Mathematics (BGSMath) }

\noindent \emph{porti@mat.uab.cat}
\end{footnotesize}

\end{document}